\documentclass[12pt]{article}

\usepackage{amsfonts}

\usepackage{amscd}
\usepackage{amsthm}
\usepackage{indentfirst}

\usepackage[all]{xy}

\usepackage{amssymb, amsmath, amsthm, amsgen, amstext, amsbsy, amsopn}
\usepackage{epsfig}
\usepackage{epic,eepic}

\newtheorem{thm}{Theorem}[section]

\newtheorem{prop}[thm]{Proposition}

\newtheorem{rem}[thm]{Remark}
\newtheorem{exam}[thm]{Example}

\newcommand{\R}{{\mathbb{R}}}

\newcommand{\T}{{\mathbb{T}}}
\newcommand{\Z}{{\mathbb{Z}}}

\def\id{{1\hskip-2.5pt{\rm l}}}

\newcommand{\Ham}{{\hbox{\it Ham\,}}}

\newcommand{\Symp}{{\hbox{\it Symp} }}

\DeclareMathOperator{\sgrad}{sgrad}

\DeclareMathOperator{\Flux}{Flux}

\begin{document}

\title{Symplectic intersections and invariant measures}

\renewcommand{\thefootnote}{\alph{footnote}}

\author{\textsc Leonid
Polterovich }

\footnotetext[1]{Partially supported by an ERC advanced grant and by an ISF grant.}

\date{\today}

\maketitle

We present a method, based on symplectic topology, which enables us to detect invariant measures with ``large" rotation vectors for a class of Hamiltonian flows. The method is robust with respect to $C^0$-perturbations of the Hamiltonian. Our results (Theorem \ref{thm-sympint} and Theorem \ref{thm-sympint-1})
are applicable even in the absence of homologically non-trivial closed orbits of the flow,
see Section \ref{sec-disc} for a discussion and an example.

The first three sections of the note deal with the case of autonomous Hamiltonians where
the formulations and the approach are more transparent. The story starts (Section \ref{sec-1})
with a ``rigid" configuration of subsets of a symplectic manifold provided by symplectic intersections theory.
In Section \ref{sec-pb} the geometric setup is reformulated in the language of function theory on symplectic manifolds by using a version of Poisson bracket invariants introduced in \cite{BEP}. With this reformulation,
the main result is proved by an elementary ergodic argument.

The last section contains a generalization to the case of non-autonomous Hamiltonians.

\section{From symplectic intersections to invariant measures}\label{sec-1}
Let $M$ be a closed manifold, and $v$ be a smooth vector field on $M$ generating a flow $\phi_t$.
For an invariant Borel probability measure $\mu$ of $\phi_t$ define its {\it rotation vector} $\rho(\mu,v) \in H_1(M,\R)$
by $$\langle a, \rho(\mu,v)\rangle := \int_M \alpha(v) d\mu\;\; \forall a \in H^1(M,\R)\;,$$
where $\alpha$ is (any) closed 1-form representing $a$ (see \cite{Schw}).

Suppose now that $(M,\omega)$ is a closed symplectic manifold. Denote by $\Symp_0(M,\omega)$ the identity
component of the symplectomorphism group of $M$. Given a path $\{\phi_t\}$, $t\in [0,1]$, $\phi_0=\id$ of symplectomorphisms, denote by $v_t$ the corresponding vector field. By the Cartan formula, $\lambda_t := i_{v_t}\omega$ is a closed $1$-form on $M$. The cohomology class $\int_0^1 [\lambda_t]$ is called {\it the flux} of the path $\{\phi_t\}$ and is denoted by $\Flux (\{\phi_t\})$. The flux does not change under a homotopy of $\{\phi_t\}$ with fixed endpoints \cite{MS}. A diffeomorphism $\theta \in Symp_0(M,\omega)$ is called Hamiltonian if it can be represented as the time one map of a path with vanishing flux.
Hamiltonian diffeomorphisms form a group denoted by $\Ham(M,\omega)$.

Our main result involves a pair of compact subsets $X,Y \subset M$ with the following properties:
\begin{itemize} \item[{(P1)}] $Y$ cannot be displaced from $X$ by any Hamiltonian diffeomorphism:
$\theta(Y) \cap X \neq \emptyset$ for every $\theta \in \Ham(M,\omega)$;
\item[{(P2)}] There exists a path $\{\psi_t\}$, $t\in [0,1]$,  $\psi_0=\id$ of symplectomorphisms
so that $\psi_1$ displaces $Y$ from $X$: $\psi_1 (Y) \cap X =\emptyset$.
\end{itemize}

\noindent
Put $X':= \psi_1(Y), a:= \Flux(\{\psi_t\})$.

\medskip
\noindent
\begin{thm}\label{thm-sympint}
For every  $F\in C^\infty(M)$ with
\begin{equation}\label{eq-ineq-main}
F|_{X} \leq 0,\;\;F|_{X'} \geq 1
\end{equation}
the Hamiltonian flow $\{\phi_t\}$ of $F$ possesses an invariant measure $\mu$ with
\begin{equation}\label{eq-mainineq}
|\langle a, \rho(\mu,\sgrad F)\rangle| \geq 1 \;.
\end{equation}
\end{thm}

\medskip
\noindent
Theorem \ref{thm-sympint} extends with minor modifications
to certain non-compact symplectic manifolds, see Remark \ref{rem-noncomp} below.
An generalization to the case of non-autonomous Hamiltonian flows is given in Section \ref{sec-non-auton}.

\medskip
\noindent
Theorem \ref{thm-sympint} is deduced from a more general statement
involving so called Poisson bracket invariants in Section \ref{sec-pb}.

\medskip

In certain situations Property (P1) can be detected by methods
of ``hard" symplectic topology.

\medskip
\noindent
\begin{exam}\label{exam-1}{\rm
Let $X=Y=L$ be a Lagrangian torus in the standard symplectic torus
$(M=\T^{2n}=\R^{2n}/\Z^{2n}, \omega= dp \wedge dq)$ given by
\begin{equation}\label{eq-lag-1}
L:=\{p_1=...=p_n=0\}\;.
\end{equation}
The fact that $X$ is non-displaceable from itself is a basic particular case of Arnold's Lagrangian intersections conjecture proved in the 1980ies by various authors including Chaperon, Hofer, Laudenbach-Sikorav, Floer (see \cite{F} and references therein).
Let $\psi_t$ be the translation by $t/2$ in the direction of $p_1$-axis.
We see that $X' = \psi_1(Y)=L'$, where
\begin{equation}\label{eq-lag-2}
L':= \{p_1=1/2, p_2=...=p_n=0\}
\end{equation}
is disjoint from $X$. The flux $a$ of $\{\psi_t\}$, $t\in [0,1]$ equals
$\frac{1}{2} \cdot [dq_1] \in H^1(\T^{2n},\R)$. We conclude that the Hamiltonian flow of any Hamiltonian
$F$ on $\T^{2n}$ with  $F|_{X} \leq 0$, $F|_{X'} \geq 1$ possesses
an invariant measure $\mu$ with
$$|\langle [dq_1], \rho(\mu,\sgrad F)\rangle| \geq 2\;.$$
Furthermore, suppose that $n=1$ and fix $\epsilon >0$. Take a function $F$ of the form $F= u(p_1)$,
so that
$$ |dq_1(\sgrad F)|= |du/dq_1| \leq 2+\epsilon\;.$$
Thus every invariant measure $\mu$ of the corresponding Hamiltonian flow satisfies
$$|\langle [dq_1], \rho(\mu,\sgrad F)\rangle| \leq 2+\epsilon\;.$$ This shows that
the conclusion of Theorem \ref{thm-sympint} is sharp. }
\end{exam}

\medskip
\noindent
\begin{exam}\label{exam-2} {\rm The previous example can be generalized as follows.
We keep notations $L$ and $L'$ for the Lagrangian tori in $\T^{2n}$ given by \eqref{eq-lag-1}
and \eqref{eq-lag-2}.  Let $(N,\Omega)$ be a closed symplectic manifold. Let $A,B \subset N$ be a pair of compact subsets such that

\medskip
\noindent
{\bf $(\spadesuit)$} $B \times L$ cannot be displaced from $A \times L$ by a Hamiltonian
diffeomorphism of $(M = N \times \T^{2n}, \Omega \oplus dp \wedge dq)$.
\medskip
\noindent

Put $X=A \times L, Y= B \times L$
and observe that the translation by $t/2$ in the direction of $p_1$-axis sends $Y$ to
$X'= B \times L'$, so that $X \cap X' = \emptyset$. Thus $X$ and $Y$ satisfy Properties (P1),(P2).

Property $\spadesuit$ holds, for instance, when $A,B$ are Lagrangian submanifolds of $N$ with non-vanishing Floer homology: $HF(A,B) \neq 0$.

Another example of $\spadesuit$ is follows:  assume that $(N,\Omega)$ splits as
$$N= N_1 \times  ... \times N_k,\;\, \Omega=\Omega_1 \oplus ... \oplus \Omega_k$$
and $$A=B= C_1 \times ... \times C_k\;, $$
where $C_j$ is the codimension one skeleton of a sufficiently fine triangulation of $N_j$.
Observe that the sets $X$ and $Y$ in this case could be quite singular.
More generally, one can take $A=B$ to be a heavy subset of $N$, see \cite{EP}.
}
\end{exam}

\section{Invariant measures vs. periodic orbits}\label{sec-disc}
It is instructive to discuss Theorem \ref{thm-sympint} in the context of the following
informal principle which nowadays is confirmed in various situations by
tools of ``hard" symplectic topology: certain robust restrictions on the $C^0$-profile of the
Hamiltonian function may yield meaningful information about homologically non-trivial
closed orbits of the Hamiltonian
flow (see e.g. \cite{LG,BPS,Lee,W,Ni}). Such a restriction in Theorem \ref{thm-sympint} is given by
inequalities \eqref{eq-ineq-main}. Observe that every $T$-periodic orbit representing a non-trivial
homology class $b \in H_1(M,\Z)$ determines an invariant measure with the rotation vector $b/T$.
Thus it is natural to ask the following question: {\it Can one, under assumptions of Theorem \ref{thm-sympint},
deduce existence of a closed orbit of the Hamiltonian flow so that the corresponding rotation vector
satisfies inequality \eqref{eq-mainineq}? } As the next example shows, the answer is in general negative.

\medskip
\noindent
\begin{exam}\label{exam-noorb}{\rm Let $M$ be the symplectic torus
$\T^4=\R^4/\Z^4$ equipped with the symplectic form
$$\omega= dp_1\wedge dq_1 +\gamma dp_2 \wedge dq_1 + dp_2 \wedge dq_2\;,$$
where $\gamma$ is an {\it irrational} number. As in Example \ref{exam-1} above, consider
Lagrangian torus $L=\{p=0\}$ and its image $L'$ under the shift by $1/2$ in $p_1$-direction.
Take a Hamiltonian $F(p,q) = \sin^2 (\pi p_1)$, so that $F=0$ on $L$ and $F=1$ on $L'$.
Exactly as in the case of the standard symplectic form, Floer theory guarantees
that Properties (P1) and (P2) hold for $X=Y=L$ and $X'=L'$. Therefore  Theorem \ref{thm-sympint}
detects an invariant measure with non-vanishing rotation number of the Hamiltonian flow of $F$.

Of course, this measure can be seen explicitly. One readily checks that the Hamiltonian vector field of $F$ is parallel to $$\frac{\partial}{\partial q_1}-\gamma \cdot \frac{\partial}{\partial q_2}\;.$$
In particular, the restriction of the Hamiltonian flow to every invariant torus $\{p_1=c_1,p_2=c_2\}$, $c_1\neq 0,1/2$ carries a quasi-periodic
motion and possesses unique invariant measure with non-vanishing rotation number. The crux of this example
is that the only closed orbits of the flow are fixed points lying on hypersurfaces $\{p_1=0\}$ and $\{p_1=1/2\}$. In particular, the flow does not have homologically non-trivial closed orbits.
}
\end{exam}

\medskip
\noindent
The discussion is continued in Remark \ref{rem-Lalonde} below.

\section{Poisson bracket invariants}\label{sec-pb}
For a closed 1-form $\alpha$ on $M$ define its locally Hamiltonian vector field $\sgrad \alpha$
by $i_{\sgrad \alpha} \omega = \alpha$. With this notation, the Hamiltonian vector field of a function $F$
is $\sgrad F:= \sgrad (-dF)$, mind the minus sign. For a function $F$ and a closed $1$-form $\alpha$
their Poisson bracket is given by $$\{F,\alpha\}=dF(\sgrad \alpha) = \alpha(\sgrad F)\;.$$
The next definition is a variation on the theme of \cite{BEP}. We write $||F||$ for the uniform norm
$\max_M |F|$.

\medskip
\noindent \begin{def} \label{def-pb}{\rm Let $X$ and $X'$ be a pair of disjoint closed subsets of $M$.
For a cohomology class $a \in H^1(M,\R)$ put
$$pb^a(X,X'):= \inf ||\{F,\alpha\}||\;,$$
where the infimum is taken over all $F\in C^\infty(M)$ with $F|_{X} \leq 0$, $F|_{X'} \geq 1$ and
all closed 1-forms $\alpha$ representing $a$. }
\end{def}

\medskip

This invariant is non-trivial, for instance, in the following situation.
Let $X,Y \subset M$ be a pair of compact subsets satisfying
Properties (P1) and (P2) of Section \ref{sec-1}: $Y$ cannot be displaced from $X$ by a Hamiltonian
diffeomorphism, but, on the other hand, it can be displaced from $X$ by a symplectomorphism
$\psi_1$ which is the end-point of a symplectic path $\{\psi_t\}$ with the flux $a \in H^1(M,\R)$.
Put $X'=\psi_1(Y)$.

\begin{prop}\label{prop-exammain} Under above assumptions, $pb^a(X,X') \geq 1$.
\end{prop}
\begin{proof} Take any closed $1$-form $\alpha$ representing $a$, and denote by $\theta_t$, $t \in [0,1]$
the corresponding locally Hamiltonian flow. Note that $\Flux(\{\psi_t^{-1}\theta_t\}) =0$, and
thus the diffeomorphism $\psi_1^{-1}\theta_1$ is Hamiltonian. Therefore
$\psi_1^{-1}\theta_1(X) \cap Y \neq \emptyset$, and thus $\theta_1(X) \cap X' \neq \emptyset$.
The latter yields existence of a point $x \in X$ such that $x':= \theta_1 x \in X'$. Given any
function $F\in C^\infty(M)$ with $F|_{X} \leq 0$, $F|_{X'} \geq 1$, we have that
$$1 \leq F(x')-F(x) = \int_0^1 \{F,\alpha\}(\theta_tx)\;dt\;.$$
Thus $||\{F,\alpha\}||\geq 1$, which proves the proposition.
\end{proof}

\medskip
\noindent
The main result of the present section is as follows.

\medskip
\noindent
\begin{thm} \label{thm-main}  Assume that $pb^{a}(X,X')=p >0$. Then
\begin{itemize} \item[{(i)}] For every  $F\in C^\infty(M)$ with $F|_{X} \leq 0$, $F|_{X'} \geq 1$,
the Hamiltonian flow $\{\phi_t\}$ of $F$ possesses an invariant measure $\mu$ with
\begin{equation}\label{eq-mainineq-1}
|\langle a, \rho(\mu,\sgrad F)\rangle| \geq p \;.
\end{equation}
\item[{(ii)}] For every closed $1$-form $\alpha$ representing $a$ the corresponding
locally Hamiltonian flow possesses a chord of time length $\leq 1/p$ joining $X$ and $X'$.
\end{itemize}
\end{thm}

\medskip
\noindent
In view of Proposition \ref{prop-exammain}, Theorem \ref{thm-sympint} immediately follows from Theorem \ref{thm-main}(i).

\medskip
\noindent
Part (ii) of Theorem \ref{thm-main} is proved exactly as in \cite[Section 4.1]{BEP}.

\medskip
\noindent
{\bf Proof of Theorem \ref{thm-main}(i):} Suppose that $pb^a(X,X')=p >0$.
Take any function $F\in C^\infty(M)$ with $F|_{X} \leq 0$, $F|_{X'} \geq 1$ and any $1$-form
$\alpha$ representing the class $a$. Denote by $\phi_t$ the Hamiltonian flow generated by $F$.
For a point $x \in M$ and a number $T >0$ denote by $\mu_{x,T}$ a probability measure on $M$ given
by
$$\int H d\mu_{x,T}:= \frac{1}{T} \cdot \int_0^T H(\phi_t x)dt\;\; \forall H \in C(M).$$
Observe, following the standard proof of the Bogolyubov-Krylov theorem (see e.g. Chapter 1, Section 8
of \cite{CFS}), that if a sequence $\mu_{x_i,T_i}$ with $T_i \to +\infty$ as $i \to +\infty$
weakly converges to a measure $\mu$, this measure is invariant under the flow $\phi_t$.

Following a suggestion by Michael Entov, apply now the averaging trick similar to the one used in \cite{BEP}.
For $T>0$ consider the averaged form
$$\alpha_T := \frac{1}{T} \cdot \int_0^T \phi_t^* \alpha\;dt\;.$$
Note that $\alpha_T$ is still a closed form in the class $a$.
Observe that
$$|\{F,\alpha_T\}(x)| = \Big{|}\int \alpha(\sgrad F) d\mu_{x,T}\Big{|}\;.$$
Since $pb^a(X,X') = p$, there exists a point $x_T \in M$ such that
\begin{equation}\label{eq-rotation-vsp}
\Big{|}\int \alpha(\sgrad F) d\mu_{x_T,T}\Big{|} \geq p\;.
\end{equation}
Compactness yields existence of a sequence $T_i \to +\infty$ so that
the sequence of measures $\mu_i := \mu_{x_{T_i},T_i}$ weakly converges to a measure $\mu$ on $M$ which,
by the above discussion, is invariant under the flow $\phi_t$. Using weak convergence and \eqref{eq-rotation-vsp} we get that
$$|\rho(\mu,\sgrad F)| = \lim_{i \to \infty} \Big{|}\int \alpha(\sgrad F) d\mu_i \Big{|} \geq p\;,$$
as required.
\qed

\bigskip
\noindent
We conclude this section with a couple of remarks.

\medskip
\noindent
\begin{rem}\label{rem-Lalonde} {\rm $\;$

\medskip
\noindent
{\bf \ref{rem-Lalonde}.1.}  It would be interesting to explore further examples of pairs $X,X' \subset M$ with
$pb^a(X,X') >0$ for some $a \in H^1(M,\R)$. A promising pool of such examples is given by
disjoint Lagrangian submanifolds $X$ and $X'$ which can be joined by pseudo-holomorphic annuli persisting
under Lagrangian (not necessarily Hamiltonian!) isotopies of $X$ and $X'$ keeping these submanifolds disjoint. In this case one can deduce positivity of $pb^a(X,X')$ by using a method of
\cite[Section 1.6]{BEP}.
The latter is based on the study of obstructions to deformations of the symplectic form $\omega$ given by $\omega_s = \omega +sdF \wedge \alpha$, where $F$ is a function which vanishes near $X$ and equals $1$ near $X'$, and $\alpha$ is a closed $1$-form representing $a$. The cohomology class $a$ is chosen in such a way that it does not vanish when evaluated on (any) boundary component of the annulus.

\medskip
\noindent
{\bf \ref{rem-Lalonde}.2.}
Let us mention that annuli with boundaries on $X$ and $X'$  satisfying a non-homogeneous Cauchy-Riemann equation play a crucial role in the Gatien-Lalonde approach to homologically non-trivial closed orbits of Hamiltonians satisfying inequalities \eqref{eq-ineq-main} (see \cite{LG,Lee}). This brings us to
a new facet of the discussion in Section \ref{sec-disc} on invariant measures vs. periodic orbits.
Assume that an appropriately defined relative Gromov-Witten invariant responsible for the count of pseudo-holomorphic annuli with boundaries on $X$ and $X'$ does not vanish. It sounds plausible that under this assumption one can establish existence of closed orbits with the rotation vector satisfying inequality \eqref{eq-mainineq-1} directly by the Gatien-Lalonde method.

\medskip
\noindent
{\bf \ref{rem-Lalonde}.3.}
 Let us mention finally that count of pseudo-holomorphic annuli appears in a number
of recent papers on the Fukaya category, see \cite{Abouzaid,Seidel}. It would be interesting to understand its applicability in our context.
}
\end{rem}

\medskip
\noindent
\medskip
\noindent
\begin{rem}\label{rem-noncomp} {\rm Theorem \ref{thm-sympint} extends with minor modifications
to certain non-compact symplectic manifolds. We start with
a compactly supported function $F$ on $M$ with $F|_X \leq 0$, $F|_{X'} \geq 1$, where $X$ and $X'$ are disjoint compact subsets. Fix a closed $1$-form $\alpha$ representing a cohomology class $a \in H^1(M,\R)$.
Assume that the locally Hamiltonian flow $\psi_t$ of $\alpha$ is well defined for all times.
Suppose now that $\psi_1(X)$ cannot be displaced from $X'$ by any Hamiltonian diffeomorphism.
Arguing exactly as in the compact case, we get that the Hamiltonian flow of $F$
possesses an invariant measure $\mu$ with compact support whose rotation vector satisfies
$$|\langle a, \rho(\mu,\sgrad F) \rangle| \geq 1\;.$$
A meaningful example is given by $M=T^*\T^n$ equipped with canonical coordinates $(p,q)$ and
the standard symplectic form $dp \wedge dq$, $X$ -- the zero section, $X'$--the Lagrangian torus
$\{p=v\}$ with $v \neq 0$ and  $\alpha = vdq$. Let us mention that for fiber-wise convex Hamiltonians
on cotangent bundles invariant measures with non-vanishing rotation vectors have been studied
in the framework of Aubry-Mather theory \cite{Mather}.
}
\end{rem}

\section{Non-autonomous Hamiltonian flows}\label{sec-non-auton}
In this section we present a generalization of Theorem \ref{thm-sympint} to general, not necessarily
autonomous, Hamiltonian diffeomorphisms. Let $M$ be a connected symplectic manifold (either open or closed).
Let $F: M \times S^1 \to \R$  a compactly supported Hamiltonian function which is time-periodic
with period $1$. Denote by $\phi_t$ the corresponding Hamiltonian flow. The time periodicity of $F$
yields $\phi_{t+1} = \phi_t \phi$, where $\phi=\phi_1$ is the time one map of the flow.
In what follows we write $F_t(x) = F(x,t)$.

For an invariant compactly supported Borel probability measure $\mu$ of  $\phi$ the {\it rotation vector} $\rho(\mu,\phi)$ is a compactly supported homology class in $H_{1,c}(M,\R)$
defined by
$$\langle a, \rho(\mu,\phi)\rangle := \int_0^1  \int_M  \langle \alpha,\sgrad F_t \rangle (\phi_tx) d\mu(x)dt$$ $$ =\int_M \left( \int_{\gamma_x}\alpha \right) \;d\mu(x) \;\; \forall a \in H^1(M,\R)\;,$$
where $\alpha$ is (any) closed 1-form representing $a$ and $\gamma_x$ stands for the trajectory $\{\phi_t x\}, t \in [0,1]$.

We claim that the rotation vector $\rho(\mu,\phi)$ depends
only on the time one map $\phi$ but not on the specific Hamiltonian $F$ generating $\phi$.
Indeed, it is a standard fact of Floer theory that if $F'$ is another Hamiltonian such that the corresponding Hamiltonian flow $\phi'_t$ satisfies $\phi'_1 =\phi$, the orbits $\gamma_x= \{\phi_tx\}$ and $\gamma'_x=\{\phi'_tx\}$, $t \in [0,1]$ are homotopic with fixed end points for every $x \in M$ and hence
$$\int_{\gamma_x}\alpha = \int_{\gamma'_x} \alpha\;.$$

Let us note that when the Hamiltonian $F$ is autonomous, $\rho(\mu,\phi)$ coincides
with $\rho(\mu, \sgrad F)$ as defined in Section \ref{sec-1}.

Consider the extended phase space $$(N,\Omega):= (M \times T^*S^1, \omega + dr \wedge ds)\;,$$
where $r$ and $s(\text{mod}\; 1)$ are canonical coordinates on $T^*S^1$. For a set
$X \subset M$ define its {\it stabilization}
$$\text{stab}(X):= X \times \{r=0\} \subset N\;.$$

The next result involves a pair of compact subsets $X,Y \subset M$
which satisfy the following properties (cf. properties (P1) and (P2) in Section \ref{sec-1}):

\begin{itemize}
\item[{(Q1)}] $\text{stab}(Y)$ cannot be displaced from $\text{stab}(X)$ by any Hamiltonian diffeomorphism of $N$.
\item[{(Q2)}] There exists a closed $1$-form $\alpha$ on $M$ whose locally Hamiltonian flow
$\{\psi_t\}$ is defined for all $t \in \R$ (this is a non-trivial assumption in
the case when $M$ is non-compact) so that $\psi_1$ displaces $Y$ from $X$.
\end{itemize}

\noindent
Put $X':= \psi_1(Y)$. Let $a$ be the cohomology class of $\alpha$.

\medskip
\noindent
\begin{thm}\label{thm-sympint-1}
For every  $F\in C^\infty(M \times S^1)$ with
$$F_t|_{X} \leq 0,\;\;F_t|_{X'} \geq 1 \;\; \forall t \in \R$$
the time one map $\phi$ of the Hamiltonian flow $\{\phi_t\}$ generated by $F$
possesses an invariant measure $\mu$ with
$$
|\langle a, \rho(\mu,\phi)\rangle| \geq 1 \;.
$$
\end{thm}

\medskip
\noindent Let us mention that properties (Q1) and (Q2) hold true in Examples \ref{exam-1}, \ref{exam-2}, \ref{exam-noorb} and in Remark \ref{rem-noncomp}.

\medskip
\noindent The proof of Theorem \ref{thm-sympint-1} starts with the Hamiltonian suspension construction:
We pass to the extended phase space $N$ and look at the {\it autonomous} Hamiltonian flow generated
by $H(x,r,s)= F(x,s) +r$. This flow encodes the original dynamics of $\phi$ on $M$ (cf. the proof of Theorem 1.12 in \cite{BEP}). Then we argue along the lines of the autonomous case considered in Theorem \ref{thm-sympint}. Let us mention however that since $H$ is not compactly supported, Theorem \ref{thm-sympint} is not directly applicable even after the modification described in Remark \ref{rem-noncomp}, which makes
the proof below somewhat more involved.

\medskip
\noindent
\begin{proof} The proof is divided into several steps.
Denote by $\pi: N \to M$ the natural projection and put $\beta:= \pi^*\alpha$.

\medskip
\noindent {\sc Step 1:} Consider a Hamiltonian $H(x,r,s) = F(x,s) +r$ on $N$.
It generates the flow
$$h_t(x(0),r(0),s(0)) = (x(t),r(t),s(t))\;,$$ where
$$x(t)= \phi_{s(0)+t}\phi_{s(0)}^{-1},\; r(t) = r(0)-\int_0^t \frac{\partial F}{\partial s}(x(t),s(t)) dt,\; s(t) = s(0) +t\;.$$ Observe that $h_t$ commutes with the symplectic $\R$-action
$$S_c: N \to N, (x,r,s) \to (x,r+c,s)\;.$$
Fix $T >0$ and put
$$\alpha_T = \frac{1}{T}\cdot \int_0^T h_t^*\beta\;.$$
Observe that
$$S_c^*h_t^*\beta = h_t^*S_c^*\beta = h_t^*\beta$$
for all $c$ and $t$, and hence
\begin{equation}
\label{eq-nonaut-complete}
S_c^*\alpha_T = \alpha_T\;.
\end{equation}

\medskip
\noindent {\sc Step 2:} We claim that the flow  $\sgrad \alpha_T$ is
defined for all times $t \in \R$. Indeed,  by \eqref{eq-nonaut-complete} the vector field
$\sgrad \alpha_T$ descends to the manifold $N':= N/\Z$ where the action of $\Z$ is given
by the integer shifts $S_k, k \in \Z$ along the $r$-direction. Furthermore, for $x$ outside
a sufficiently large compact in $M$ we have that $h_t(x,r,s)= (x,r, s+t)$ and hence
$\alpha_T = \beta = \pi^*\alpha$. It follows, by our assumption on $\alpha$, that the trajectories
of $\sgrad \alpha_T$ cannot escape to infinity in finite time on $N'$. Thus the flow
of  $\sgrad \alpha_T$ is defined for all times on $N'$, and therefore it lifts to a well
defined flow $\theta_t$, $t \in \R$ on $N$. The claim follows.

\medskip
\noindent {\sc Step 3:}
Since $[\alpha_T]= [\beta] \in H^1(N,\R)$,
the flow
$$g_t: = (\psi_t \times \id)^{-1} \theta_t : N \to N$$
is Hamiltonian. By Property (Q1) $g_t (\text{stab}(X)) \cap \text{stab}(Y) \neq \emptyset$ which yields
$$\theta_1(\text{stab}(X)) \cap \text{stab}(X') \neq \emptyset\;.$$
Thus there exists a point $z \in \text{stab}(X)$ such that $\theta_1(z) \in \text{stab}(X')$.
Observe that $H \leq 0$ on $\text{stab}(X)$ and $H \geq 1$ on $\text{stab}(X')$. Since $$1 \leq {H}(\theta_1z)-{H}(z) = \int_0^1 \{H,\alpha_T\}(\theta_tz)\;dt\;,$$
there exists $t' \in [0,1]$ with $$|\{{H},\alpha_T\}(\theta_{t'} z)| \geq 1\;.$$
Put $y=\theta_{t'}z$.

Note that $S_c^*H=H+c$ and $S_c^*\alpha_T = \alpha_T$ by \eqref{eq-nonaut-complete}.
It follows that $S_c^*\{H,\alpha_T\} = \{H,\alpha_T\}$, and hence
$$|\{{H},\alpha_T\}(S_cy)| \geq 1\;\; \forall c \in \R\;.$$
Therefore there exists a point $y_T := (x_T,0, s_T)$ such that
\begin{equation}
\label{eq-nonauton-brack}
|\{{H},\alpha_T\}(y_T)| \geq 1\;.
\end{equation}
As we shall see below, it matters that the $r$-coordinate
of $y_T$ vanishes.

\medskip
\noindent {\sc Step 4:} We claim that all the orbits $\gamma_T := \{h_ty_T\}$, $t \in [0,T]$
lie in some compact $Q \subset N$ for all $T >0$. Denote by $K \subset M \times S^1$ the support
of $F(x,s)$. The set $$X:= (M \setminus K) \times \R \subset N$$
is invariant under $h_t$. Moreover, on $X$ we have that
$\{H,\alpha_T\}(z) = \{r,\beta\} = 0$
which violates \eqref{eq-nonauton-brack}. It follows that $\gamma_T \subset K \times \R$.
Furthermore, write $h_ty_T = (x^t,r^t,s^t)$. By the energy conservation law,
$$F(x^t,s^t) + r^t = F(x_T,s_T) +0\;,$$
which yields an upper bound
$$|r_t| \leq C:= \max F - \min F \;\;\forall t\;.$$
The claim follows with $Q= K \times [-C,C]$.

\medskip
\noindent {\sc Step 5:} Define a measure $\nu_T$ on $N$ by
$$\int G d\nu_T = \frac{1}{T} \int_0^T G(h_t y_T)\; dt\;\; \forall G \in C(N)\;.$$
By Step 4, these measures are supported in the compact subset $Q$. Hence,
after passing to a subsequence $T_k \to +\infty$, they weakly converge to a measure, say
$\nu$ on $N$. The standard Bogolyubov-Krylov argument shows that $\nu$ is $h_t$-invariant.
Furthermore, by \eqref{eq-nonauton-brack}
$$\Big{|} \int \beta(\sgrad H) d\nu_T \Big{|} = |\{H,\alpha_T\}(y_T)| \geq 1\;,$$
and hence
\begin{equation}\label{eq-nonaut-main}
\Big{|} \int \beta(\sgrad H) d\nu \Big{|}  \geq 1\;.
\end{equation}

\medskip
\noindent {\sc Step 6:} Consider the flow
$$g_t: M \times S^1 \to M \times S^1,\; (x,s) \to (\phi_{s+t}\phi_s^{-1}x, s+t)\;.$$
Denote by $\tau: N \to M \times S^1$ the natural projection. Since $\tau h_t = g_t$,
the push-forward measure $\sigma:= \tau_*\nu$ is invariant under $g_t$. Further,
$\beta(\sgrad H) = \alpha (\sgrad F_s)$ is independent on $r$. Therefore
inequality \eqref{eq-nonaut-main} yields
\begin{equation}\label{eq-nonaut-main-1}
\Big{|} \int \alpha(\sgrad F_s) d\sigma \Big{|}  \geq 1\;.
\end{equation}

\medskip
\noindent {\sc Step 7:} Invariant measures of the flow $g_t$ have quite a simple structure. To see this,
introduce diffeomorphisms $A$  and $B$ of $M \times \R$ given by
$$A: (x,s) \mapsto (\phi_sx, s), \;\; B: (x,s) \mapsto (\phi^{-1}x, s+1)\;,$$
and the translation $R_t$ of $M \times \R$,
$$R_t:  (x,s) \mapsto (x,s+t)\;.$$
Let $$\widetilde{g}_t (x,s) = (\phi_{s+t}\phi_s^{-1}x, s+t)$$
be the lift of the flow $g_t$ to the cover $M \times \R$.
Borel probability measures $\sigma$ on $M \times S^1$ are in one to one correspondence
with $R_1$-invariant Borel measures $\widetilde{\sigma}$ on $M \times \R$ satisfying
$\widetilde{\sigma}(M \times [0,1)) =1$. The measure $\sigma$ is $g_t$-invariant
if and only if $\widetilde{\sigma}$ is $\widetilde{g}_t$-invariant.
Observe that $\widetilde{g}_t = AR_tA^{-1}$. Thus every invariant measure  $\widetilde{\sigma}$
of $\widetilde{g}_t$ has the form
$A_*\overline{\mu}$, where $\overline{\mu}$ is an invariant measure of the flow $R_t$.
Note that $\overline{\mu}$ is necessarily of the form $\mu \otimes ds$, where $ds$ is the Lebesgue
measure on $\R$ and $\mu$ a measure on $M$.

The measure $A_*\overline{\mu}$ is $R_1$-invariant if and only
if $$A^{-1}_*R_{1*}A_*\overline{\mu}= \overline{\mu}\;.$$
Since $A^{-1}R_1A = B$, we have that $B_*(\mu \otimes ds) = (\mu \otimes ds)$,
which is equivalent to the fact that $\mu$ is $\phi$-invariant.

Returning back to $M \times S^1$, we conclude that every $g_t$-invariant Borel probability
measure $\sigma$ on $M \times S^1$ satisfies
$$ \int_{M \times S^1} G(x,s)\; d\sigma(x,s) = \int_0^1 \int_M G(\phi_s x,s)\;d\mu(x)ds\;\;\forall G \in C(M \times S^1),$$
where $\mu$ is a $\phi$-invariant Borel probability
measure on $M$.

\medskip
\noindent {\sc Step 8:} Apply the conclusion of Step 7 to the measure $\sigma$ constructed in
Step 6. Inequality \eqref{eq-nonaut-main-1} reads
$$\Big{|} \int_0^1 \int \langle \alpha, \sgrad F_s\rangle (\phi_sx)\; d\mu(x)ds  \Big{|}  \geq 1\;,$$
where the measure $\mu$ on $M$ is $\phi$-invariant. By definition, this means that $|\langle a, \rho(\mu,\phi)\rangle|>1$. This completes the proof.
\end{proof}

\medskip
\noindent
{\bf Acknowledgements.} This note was inspired by Hofer's 2013 Aisenstadt Lectures in the CRM, Montreal.
I am grateful to Michael Entov for comments which led to a drastic simplification of the original proof (based on \cite{Sul}) of the main theorem, as well as to a generalization to the non-autonomous case.
I thank Lev Buhovsky and Helmut Hofer for useful discussions.


\begin{tabular}{l}
Leonid Polterovich\\
School of Mathematical Sciences\\
Tel Aviv University\\
Tel Aviv 69978, Israel\\
polterov@runbox.com\\
\end{tabular}


\begin{thebibliography}{99}

\bibitem{Abouzaid} Abouzaid, M., {\it A geometric criterion for generating the Fukaya category,}
Publ. Math. Inst. Hautes \'{E}tudes Sci. {\bf 112} (2010), 191–-240.

 \bibitem{BPS}  Biran, P., Polterovich, L., Salamon, D., {\it Propagation in Hamiltonian dynamics and relative symplectic homology,} Duke Math. J. {\bf 119} (2003), 65–-118.

\bibitem{BEP} Buhovsky, L., Entov, M., Polterovich, L., {\it Poisson brackets and symplectic invariants}, Selecta Math. (N.S.) {\bf 18} (2012), 89–-157.

    \bibitem{CFS} Cornfeld, I.P., Fomin, S.V., Sinai Ya.G.,
    {\it Ergodic Theory,} Springer, 1982.

     \bibitem{LG} Gatien, D., Lalonde, F., {\it Holomorphic cylinders with Lagrangian boundaries and Hamiltonian dynamics,} Duke Math. J. {\bf 102} (2000), 485–-511.

        \bibitem{EP} Entov, M., Polterovich, L., {\it Rigid subsets of symplectic manifolds,}
 Compos. Math. {\bf 145} (2009), 773–-826.


    \bibitem{Fischer} Fischer, T., {\it Existence, uniqueness, and minimality of the Jordan measure decomposition}, Preprint arXiv:1206.5449, 2012.

\bibitem{F} Floer, A., {\it Morse theory for Lagrangian intersections,}
J. Differential Geom. {\bf 28} (1988), 513–-547.

\bibitem{Lee} Lee, Y.-J., {\it Non-contractible periodic orbits, Gromov invariants, and Floer-theoretic torsions,} Preprint arXiv:math/0308185, 2003.

    \bibitem{Mather} Mather, J., {\it Action minimizing invariant measures for positive definite Lagrangian systems,} Math. Z. {\bf 207} (1991), 169–-207.

\bibitem{MS}  McDuff, D., Salamon, D., \textit{Introduction to Symplectic Topology, second edition},
Oxford University Press, New York, 1998.

\bibitem{Ni}  Niche, C., {\it Non-contractible periodic orbits of Hamiltonian flows on twisted cotangent bundles,} Discrete Contin. Dyn. Syst. {\bf 14} (2006), 617–-630.


\bibitem{Schw} Schwartzman, S., {\it Asymptotic cycles,} Ann. of Math. {\bf 66} (1957), 270–-284.


\bibitem{Seidel} Seidel, P., {\it Disjoinable Lagrangian spheres and dilations,}
Preprint  arXiv:1307.4819, 2013.

\bibitem{Sul}  Sullivan, D., {\it Cycles for the dynamical study of foliated manifolds and complex manifolds,} Invent. Math. {\bf 36} (1976), 225–-255.

    \bibitem{W} Weber, J., {\it Noncontractible periodic orbits in cotangent bundles and Floer homology},
Duke Math. J. {\bf 133} (2006), 527–-568.

\end{thebibliography}
\end{document}